\documentclass[11pt]{amsart}
\usepackage{pstricks,pst-plot, amsmath, amssymb, }

\definecolor{Black}{cmyk}{0,0,0,1}
\definecolor{OrangeRed}{cmyk}{0,0.6,1,0}            
\definecolor{DarkBlue}{cmyk}{1,1,0,0.20}
\definecolor{myblue}{rgb}{0.66,0.78,1.00}
\definecolor{Violet}{cmyk}{0.79,0.88,0,0}
\definecolor{Lavender}{cmyk}{0,0.48,0,0}
\parskip=\smallskipamount

\newtheorem{theorem}{Theorem}[section]

\newtheorem{corollary}[theorem]{Corollary}

\theoremstyle{definition}

\newtheorem{defn}[theorem]{Definition}

\newcommand{\C}{\mathbb{C}}

\newcommand{\bea}{\begin{eqnarray*}}
\newcommand{\eea}{\end{eqnarray*}}

\newcommand{\htop}{h_{\operatorname{top}}}

\newcommand{\ra}{\rightarrow}

\numberwithin{equation}{section}

\begin{document}
\title[Infinite entropy for transcendental functions]{Infinite entropy for transcendental entire functions with an omitted value}

\author[A.M. Benini]{Anna Miriam Benini$^{\dag}$}
\author[J.E.  Forn{\ae}ss ]{John Erik Forn{\ae}ss$^{*}$}
\author[H. Peters]{Han Peters}


\date{\today}


\thanks{$^{\dag}$  Supported by the by  the  Horizon 2020 Marie Curie IEF grant n. 703269 COTRADY}
\thanks{$^{*}$  Supported by the NFR grant no. 10445200}
\address{ A.M. Benini: Departamento de  Matem\'atica y Inform\'atica\\
Universitat de Barcelona    \\
 Spain} \email{ambenini@gmail.com}
\address{ H. Peters: Korteweg de Vries Institute for Mathematics\\
University of Amsterdam\\
the Netherlands} \email{hanpeters77@gmail.com}
\address{ J.E. Fornaess: Department of Mathematical Sciences\\
NTNU Trondheim, Norway} \email{john.fornass@ntnu.no}

\dedicatory{Dedicated to L\^{e}  V\u{a}n Thi\^{e}m on the occasion of his centenary.}

\begin{abstract}
We prove that entire transcendental holomorphic functions with an omitted value have infinite entropy. A proof for general transcendental entire functions will be given in an upcoming paper.
\end{abstract}

\maketitle

 \section{Introduction}

Understanding   how many times a transcendental meromorphic functions takes a given value on a circle of radius $r$ is one of the fundamental problems in the deep theory of Nevanlinna.
In one of his earliest papers \cite{CVT1949}, Le Van Thiem worked on the distribution of values of meromorphic functions in one complex variable, partially solving the inverse problem in the value distribution theory for meromorphic functions. In this paper we prove results regarding the distribution of values for entire functions with an omitted singular value, which give as corollary that such functions have infinite entropy.

By Picard's theorem, an entire transcendental function $f$ takes every value in $\C$ infinitely many times, except for at most one value, which, if it exists, is called the exceptional value. When the exceptional value is not taken at all, it is called an omitted value. By Montel's Theorem, if $z_0$ is a repelling periodic point and $U$ is an arbitrary neighborhood of $z_0$ then the family $\{f^n\}$ takes infinitely times each value except for the exceptional value. On the other hand, Montel's Theorem gives no quantitative estimates on how many iterates are needed  before a given value is taken again.
 Wiman Valiron theory tells that, for each $n$ and  for all values of $r>0$ sufficiently large outside  a set of finite logarithmic measure, there exists at least one point of modulus $r$, and   a neighborhood of it, whose image covers   a larger annulus at least $n$ times. However, no information is given about whether and when the image of the annulus gets back to covering the original neighborhood.

Our main theorem is the following.

\begin{theorem}\label{thm:main theorem}
Let $f$ be a transcendental entire function with an omitted value $\beta$, and let $n \in \mathbb N$. For $R>0$ define the annulus
$$
A_R:=\{R/2<|z-\beta|<2R\}.
$$
Then there exists $\delta>0$ and  $R=R(n)$ such that $f(A_R)\supset A_R$ and every point in $A_R$ has at least $n$ preimages in $A_R$ which are at Euclidean distance at least $\delta$ from each other.
\end{theorem}

Theorem~\ref{thm:main theorem} has the following corollary:

\begin{theorem}\label{thm:Infinite entropy}
Let $f$ be a transcendental entire function with an omitted value. Then $f$ has infinite topological entropy.
\end{theorem}

The fact that   entire transcendental functions should have infinite entropy is no surprise. Indeed, it has been well known for several decades that rational maps of degree $d$ acting on the Riemann sphere  have topological entropy equal to $\log d$ (see \cite{MisiurewiczPrzytycki} and \cite{Gromov}, and independently  \cite{lyubich}).

One of the reasons why the problem of topological entropy for entire transcendental maps has not been addressed for so long is that there are several non equivalent definitions of topological entropy on non-compact metric spaces; we refer to  \cite{Hasselblatt} and the references therein for the definition of topological entropy on noncompact spaces. See also the book \cite{Wal82} for a background in ergodic theory).  Observe  that transcendental maps are not uniformly continuous on $\C$ and that they do not extend continuously to its one-point compactification, that is the Riemann sphere  $\hat{\C}$. The definition of topological entropy that we will use is the following.
\begin{defn}[Definition of topological entropy]\label{defn:entropy noncompact}
Let $f:Y \rightarrow Y$ be a self-map of a metric space $(Y, d)$.
Let $X$ be a compact subset of $Y.$ Let   $n \in \mathbb N$ and $\delta > 0$.
A  set $E\subset X$ is called \emph{$(n,\delta)$-separated}  if \begin{itemize}
\item for any  $z\in E$, its orbit $\{z, f(z),\dots, f^{n-1}(z)\}\subset X$;
\item  for any $z\neq w\in E$ there exists $k\leq n-1$ such that
$d(f^k(z),f^k(w))> \delta$.
\end{itemize}
Let $K(n, \delta)$ be the maximal cardinality of an $(n,\delta)$-separated set.
Then the \emph{topological entropy} $\htop(X,f)$ is defined as
$$
\htop(X,f):=\sup_{\delta>0}\left\{\limsup_{n\ra\infty}\frac{1}{n}\log K(n,\delta)\right\}.
$$
We define the topological entropy  $\htop(f)$ of $f$ on $Y$ as the supremum of $\htop(X,f)$ over all compact subsets $X\subset Y.$
\end{defn}

In general, this definition depends on the metric $d$. In our setting, the natural metrics on $\C$ with respect to the dynamics of transcendental entire functions are the spherical metric and the Euclidean metric. Since they are comparable on compact subsets of $\C$, both choices yield the same result with respect to Definition~\ref{defn:entropy noncompact}.
Observe that $\C$ endowed with the spherical metric  is totally  bounded and  that   our conditions for being an $(n,\delta)$-separated set are more restrictive than the conditions used to defined the Bowen compacta entropy in paragraph 2.1 in \cite{Hasselblatt}. In view of Theorem 6 there, the topological entropy is infinite also according to the other two definitions presented in \cite{Hasselblatt}.

\section{Infinite entropy for functions with no zeroes}

For $f$ entire transcendental and $R>0$ let $$M(R,f):=\sup_{|z|=R}|f(z)|$$ denote the maximum modulus of $f$.
Recall that
\begin{equation}\label{eqtn:growth of M}
\lim_{R\ra\infty}
\frac{\log M(R,f)}{\log R}=\infty.
\end{equation}

\begin{proof}[Proof of Theorem~\ref{thm:main theorem}]
Up to considering conjugation by a translation, we can assume that $\beta=0$ and hence that  $f:\C\ra\C\setminus\{0\}.$

We consider values $R>0$ sufficiently large for which there exist points $w_m, w_M \in \{|z| = R\}$ with $|f(w_m)| < 1$ and $|f(w_M)| > R^{2n}$. Since $f \neq 0$ there exists an entire function $g$ for which $f = g^n$. It follows that $|g(w_m)| < 1$ and $|g(w_M)| > R^2$.

Define the annulus
$$
A_R^{1/n} = \{\left(\frac{R}{2}\right)^\frac{1}{n} < |z| < \left(2R\right)^{\frac{1}{n}}\}.
$$
Suppose for the sake of a contradiction that $A_R^{1/n} \nsubseteq g(A_R)$. Write $\zeta$ for a point in $A_R^{1/n} \setminus g(A_R)$ and define the holomorphic function $h = g/\zeta$. Then $h$ maps $A_R$ into $\mathbb C \setminus \{0,1\}$, a hyperbolic Riemann surface.

Note that the hyperbolic distance between $w_m$ and $w_M$ in the domain $A_R$ is bounded from above by a constant independent of $R$. However, for $R$ sufficiently large it follows that $|h(w_M)|> R$, while $|h(w_m)| < 1$. By completeness of the hyperbolic metric it follows that the distance between $h(w_M)$ and $h(w_m)$ converges to infinity as $R \rightarrow \infty$. Since holomorphic maps cannot increase hyperbolic distances, we obtain a contradiction when $R$ is sufficiently large.

Thus we have proved that $A_R^{1/n} \subseteq g(A_R)$ when $R$ is sufficiently large, and thus that $A_R \subseteq f(A_R)$. The fact that there exist at least $n$ distinct preimages, with uniform bounds on the distance between these $n$ points, follows by considering the $n$ distinct branches for the function $g = f^{1/n}$.
\end{proof}

\begin{corollary}
A transcendental function omitting $0$ has infinite entropy.
\end{corollary}
\begin{proof}
Let $n \in \mathbb N$ and let $R>0$ as constructed in the theorem. Let $\delta>0$ be such that any point $w \in A_R$ has at least $n$ preimages $z_1, \ldots , z_n$ satisfying $|z_i - z_j| > \delta$ for $i \neq j$. Let $y \in A_R$ and $E_0 = \{y\}$, and recursively choose finite subsets $E_j \subset A_R$ with the property that $E_{j+1}$ contains for each $w\in E_j$ exactly $n$ preimages $z_1, \ldots , z_n$ with $|z_i - z_j| > \delta$ for $i \neq j$, and no other points. By construction $E_j$ is $(j,\delta)$ separated, and the cardinality of $E_j$ is exactly $n^j$. It follows that
$$
h_{top}(\overline{A_R}, f) \ge \log(n),
$$
which implies that the topological entropy of $f$ is infinite.
\end{proof}

\bibliographystyle{amsalpha}
\bibliography{entropyviet}

\providecommand{\bysame}{\leavevmode\hbox to3em{\hrulefill}\thinspace}
\providecommand{\MR}{\relax\ifhmode\unskip\space\fi MR }
\providecommand{\MRhref}[2]{%
  \href{http://www.ams.org/mathscinet-getitem?mr=#1}{#2}
}
\providecommand{\href}[2]{#2}
\begin{thebibliography}{HNP08}

\bibitem[Gro03]{Gromov}
Mikha{\"i}l Gromov, \emph{On the entropy of holomorphic maps}, Enseign. Math.
  (2) \textbf{49} (2003), no.~3-4, 217--235. \MR{2026895}

\bibitem[HNP08]{Hasselblatt}
Boris Hasselblatt, Zbigniew Nitecki, and James Propp, \emph{Topological entropy
  for nonuniformly continuous maps}, Discrete Contin. Dyn. Syst. \textbf{22}
  (2008), no.~1-2, 201--213. \MR{2410955}

\bibitem[Lju83]{lyubich}
Mikhail~Ju. Ljubich, \emph{Entropy properties of rational endomorphisms of the
  riemann sphere}, Ergodic Theory Dynam. Systems \textbf{3} (1983), no.~3,
  351--385. \MR{741393}

\bibitem[MP77]{MisiurewiczPrzytycki}
Micha\l\ Misiurewicz and Feliks Przytycki, \emph{Topological entropy and degree
  of smooth mappings}, Bull. Acad. Polon. Sci. S\'er. Sci. Math. Astronom.
  Phys. \textbf{25} (1977), no.~6, 573--574. \MR{0458501}

\bibitem[Thi49]{CVT1949}
Le-Van Thiem, \emph{\"uber das {U}mkehrproblem der {W}ertverteilungslehre},
  Comment. Math. Helv. \textbf{23} (1949), 26--49. \MR{0030609}

\bibitem[Wal82]{Wal82}
Peter Walters, \emph{An introduction to ergodic theory}, Graduate Texts in
  Mathematics, vol.~79, Springer-Verlag, New York-Berlin, 1982. \MR{648108}

\end{thebibliography}

\end{document}